\documentclass [11pt]{article}
\usepackage{amsthm}
\usepackage{amsmath}
\usepackage{amssymb}
\usepackage{enumerate}
\usepackage{epsfig}
\usepackage{srcltx}
\usepackage{graphics}

\newcounter{contador}
\newcounter{contador2}

\newtheorem{proposition}[contador]{Proposition}
\newtheorem{theorem}[contador]{Theorem}
\newtheorem{lemma}[contador]{Lemma}

\newtheorem{remark}[contador]{Remark}

\newcommand{\C}{{\mathbb C}}
\newcommand{\N}{{\mathbb N}}

\newcommand{\pr}{{P}{\mathbb C^2}}

\newcommand{\pru}{{P}{\mathbb C^1}}

\newcommand{\Z}{{\mathbb Z}}

\newcommand{\homog}{[x_0:x_1:x_2]}

\title{Zero entropy for some birational maps of $\C^2$\footnote{{\bf Acknowledgements}.
The first author is supported by Ministry of Economy, Industry and Competitiveness of the Spanish Government through grants MINECO/FEDER MTM2016-$77278$ and also supported by the grant $2014$-SGR-$568$ from AGAUR, Generalitat de Catalunya. The GSD-UAB Group is supported by the Government of Catalonia through the SGR program. It is also supported by MCYT through the grant MTM$2008-03437.$}}

\author{Anna Cima$^{(1)}$ and Sundus Zafar$^{(1)}$
  \\*[.1truecm]
{\small \textsl{$^{(1)}$ Dept. de Matem\`{a}tiques, Facultat de
Ci\`{e}ncies,}}
\\*[-.25truecm] {\small \textsl{Universitat Aut\`{o}noma de Barcelona,}}
\\*[-.25truecm] {\small \textsl{08193 Bellaterra, Barcelona, Spain}}
\\*[-.25truecm] {\small \textsl{cima@mat.uab.cat,
sundus@mat.uab.cat}}}

\date{}

\begin{document}
\maketitle



\begin{abstract}
This work deals with a special case of family of birational maps $f:\C^2 \to \C^2$ dynamically classified in \cite{CZ2}. In this work we study the zero entropy sub families of $f$. The sequence of degrees $d_n$ associated to the iterates of $f$ is found to grow periodically, linearly, quadratically or exponentially. Explicit invariant fibrations for zero entropy families and all the integrable and periodic mappings inside the family $f$ are given.

\end{abstract}




\section{Introduction}

Consider the family of fractional maps $f:\C^2 \to \C^2$ of the
form:
\begin{equation}
\label{eq1}
 f(x,y) = \left( {\alpha _0} + {\alpha _1}x + {\alpha
_2}y,\frac{{\beta _0} + {\beta _1}x + {\beta _2}y}{{\gamma _0} +
{\gamma _1}x + {\gamma _2}y} \right)\,,\,(\gamma_1,\gamma_2)\ne(0,0)
\end{equation}
where the parameters $\alpha_i,\, \beta_i,\, \gamma_i,\,\, i \in \{0,1,2\}$ are complex numbers.

In this work the family of mappings $f(x,y)$ in (\ref{eq1}) is
required to be birational in general. The values of parameters
$\alpha_i,\, \beta_i,\, \gamma_i,\,\, i \in \{0,1,2\}$ for which
$f(x,y)$ is a birational mapping is discussed in Lemma
\ref{conditions} in this article. The study of the dynamics
generated by birational mappings in the plane and their
classification is a well discussed topic in recent years as can be
found in \cite{ADMV,BK2,BK3,CGM2,CGM3,JD,DS,LG,LKP,LGK,LKMR,PRo,Ro1,ZE}. The family of
mappings $f(x,y)$ in (\ref{eq1}) is dynamically classified in
\cite{CZ2}. In this work we are going to study the mapping $f$ in
case where it shows a kind of degenerate behavior for general values
of parameters.

For a birational map $f(x,y)$ the sequence of degrees $d_n$ of the iterates of $f$ satisfies a homogeneous linear recurrence, see \cite{DF}. This is governed by the characteristic polynomial $\mathcal{X}(x)$ of a certain matrix associated to $F,$ where $F:{\pr} \to {\pr}$ is the extension of $f:\C^2 \to \C^2$ in the projective plane $\pr.$ This further provides information regarding the quantity called as the \textit{dynamical degree of $F$} and defined as
\begin{equation}\label{eq.1}
\delta(F): = \mathop {\lim }\limits_{n \to \infty } {\left( {\deg
        ({F^n})} \right)^{\frac{1}{n}}},
\end{equation}
where $F^n$ represents the iterates of $F.$ The logarithm of $\delta(F)$ is the \textit{algebraic entropy of $F$}, see \cite{BK2,BK3,BV,JD,DF}.

Considering the embedding $(x_1,x_2)\in\C^2 \mapsto [1:x_1:x_2] \in
{\pr}$ into projective space, the induced map $F:{\pr} \to {\pr}$
has three components $F_i[x_0:x_1:x_2]\,,\,i=1,2,3$ which are
homogeneous polynomials as $F[{x_0}:{x_1}:{x_2}] = [F_1\homog:F_2\homog:F_3\homog],$ where
\begin{equation}
    \label{EQ2}
    \begin{array}{l}
        F_1\homog=x_0({\gamma _0}x_0 + {\gamma _1}x_1 + {\gamma _2}x_2),\\
        F_2\homog=({\alpha _0}x_0 + {\alpha _1}x_1 + {\alpha _2}x_2)({\gamma_0}x_0 + {\gamma _1}x_1 + {\gamma _2}x_2),\\
        F_3\homog=x_0({\beta _0}x_0 + {\beta _1}x_1 + {\beta _2}x_2).
    \end{array}
\end{equation}

The map $F$ has degree two as the components of $F$ do not have common factors for general values of the parameters. Similarly the degree of each iterate of $F$ can be found in general by iterating $F$ and removing the common homogenous components as ${F^n} = F \circ \cdot \cdot \cdot  \circ F$ for each $n\in\N.$

Birational mappings $F:{\pr} \to {\pr}$ have an indeterminacy set
${\mathcal{I}(F)}$ of points where $F$ is ill-defined as a
continuous map. Hence they also have a set of curves which are sent
to a single point called the {\it exceptional locus} of $F$ denoted
as $\mathcal{E}(F).$ Generically the mappings of the form
(\ref{eq1}) have three indeterminacy points. The exceptional locus
is formed by three straight lines, each two of them intersecting on
a single indeterminate point of $F$. We call them as \textit{non
degenerate mappings}. However in some cases exceptional locus is
formed by only two straight lines. In this case these mappings are
identified as {\it degenerate mappings}. Lemma \ref{conditions} in
preliminary results section discusses the conditions for
birationality and degeneracy of family $f$ in (\ref{eq1}). This
study includes all the subfamilies of $f(x,y)$ where it shows a
degenerate behavior. The cases where exceptional locus is formed by
three straight lines are discussed and studied in the papers
\cite{CZ1} and \cite{CZ2}. Such cases are recognized as \textit{non
degenerate mappings}.

The first goal of this study is look for sequence of degrees $d_n.$
This is done by performing a series of blow-up's in order to find
the characteristic polynomial which determines the behaviour of
$d_n.$

The second goal is to identify for which values of the parameters
these mappings have zero algebraic entropy and extract dynamical
consequences. For this we use the results of Diller and Favre, see
\cite{DF}, which characterize the growth rate of $d_n$ with the
existence of invariant fibrations. We find all the prescribed
invariant fibrations in each one of this cases. We emphasize which
elements of the family are integrable mappings. We also distinguish
all the periodic mappings giving a pair of first integrals
generically transverse.

The article is organized as follows: Section two is devoted to give
some preliminary results on birational maps and family $f$, in order to
describe the blow-up process and the Picard group. In Section three we study
 the subfamily $\alpha_1\gamma_2-\alpha_2\gamma_1=0,$ while in Section four we study
 the subfamily $\alpha_1\beta_2-\alpha_2\beta_1=0.$

The results that we get are the following. We have named Theorem the
results on the dynamical degree and the growth of $d_n$ and
Proposition the results on the zero entropy and existence of
invariant fibrations. In this way in Section 3 we give Theorem
\ref{CD2} with Proposition \ref{zeroentropyone} concerning the
family $\alpha_1\gamma_2-\alpha_2\gamma_1=0.$ Section 4 which deals
with mappings satisfying $\alpha_1\beta_2-\alpha_2\beta_1=0,$ is
splitted in three subsections. We present Theorem \ref{CD3} with
Proposition \ref{zeroentropytwo} when $\gamma_1\gamma_2\ne 0,$
Theorem \ref{gamma1} with Proposition \ref{zeroentropythreee}
($\alpha_2\ne 0$) and Proposition \ref{zeroentropyfourr} ($\alpha_2=0$) when
$\gamma_1=0$ and Theorem \ref{gamma2} with Proposition
\ref{zeroentropyfive} when $\gamma_2=0.$

\section{Preliminary results}

Consider the mapping $F[x_0:x_1:x_2]:{\pr} \to {\pr}$ in (\ref{EQ2}), then the exceptional locus of $F[x_0:x_1:x_2]$ is given as
${\mathcal{E}(F)}\,=\,\{S_0,S_1,S_2\},$ where
\begin{equation*}
S_0=\{x_0=0\},\quad S_1=\{\gamma_0 x_0+\gamma_1 x_1+\gamma_2
x_2=0\},
\end{equation*}
\begin{equation*}
S_2=\{\left(\alpha_1
(\beta\gamma)_{02}-\alpha_2 (\beta\gamma)_{01}\right)\,x_0+\alpha_1
(\beta\gamma)_{12} x_1+\alpha_2 (\beta\gamma)_{12} x_2=0\}.
\end{equation*}
We have used the notation:
$(\delta\epsilon)_{ij}=\delta_i\epsilon_j-\delta_j\epsilon_i.$ The
exceptional locus of $F^{-1}[x_0:x_1:x_2]$ is
${\mathcal{E}(F^{-1})}\,=\,\{T_0,T_1,T_2\},$ where
$$\begin{array}{l}
T_0=\left\{\left(\gamma_0 (\alpha\beta)_{12}-\gamma_1 (\alpha\beta)_{02}+\gamma_2 (\alpha\beta)_{01}\right)x_0-(\beta\gamma)_{12} x_1=0\right\},\\
T_1=\{(\alpha\beta)_{12} x_0-(\alpha\gamma)_{12} x_2=0\},\quad\quad
T_2=\{x_0=0\}.
\end{array}$$

The birational map $F[x_0:x_1:x_2]$ has an indeterminacy set
${\mathcal{I}(F)}$ of points where $F$ is ill-defined as a
continuous map. This set is given by:
$$\{\homog\in\pr:F_1\homog=0,F_2\homog=0,F_3\homog=0]\},$$ which
gives:
$${\mathcal{I}(F)}\,=\,\{O_1,O_2,O_3\},$$ where
$$\begin{array}{l}
O_0=[(\beta\gamma)_{12}:(\beta\gamma)_{20}:(\beta\gamma)_{01}],\\
O_1=[0:\alpha_2:-\alpha_1],\\
O_2=[0:\gamma_2:-\gamma_1],
\end{array}$$
and $(\beta\gamma)_{ij}:=\beta_i\gamma_j-\gamma_j\beta_i$ for
$i,j=0,1,2.$

By calling $g(x,y)$ the inverse of $f(x,y)$ given in (\ref{eq1}) and  considering $G\homog$ its extension on $\pr,$ also a indeterminancy set ${\mathcal{I}(G)}$
exists i.e. ${\mathcal{I}(G)}\,=\,\{A_1,A_2,A_3\},$ where
$$\begin{array}{l}
A_0=[0:1:0],\\
A_1=[0:0:1],\\
A_2=[(\beta\gamma)_{12}\,(\alpha\gamma)_{12},(\alpha_0\,(\beta\gamma)_{12}-\alpha_1\,(\beta\gamma)_{02}+
\alpha_2\,(\beta\gamma)_{01})\,(\alpha\gamma)_{12}:(\alpha\beta)_{12}\,(\beta\gamma)_{12}].
\end{array}$$

We are interested in the birational mappings (\ref{eq1}) when the corresponding $F$ only has two distinct exceptional curves. Next lemma informs about the set of parameters which are available in this study.

Recall that a birational map is a map $ f:\C^2\rightarrow \C^2$ with
rational components such that there exists an algebraic curve $V$
and another rational map $g$ such that $f\circ g=g\circ f=id$ in
$\C^2\setminus V.$

\begin{lemma}\label{conditions}
    Consider the mappings $$f(x_1,x_2)=\left( {\alpha _0} + {\alpha
        _1}x_1 + {\alpha _2}x_2,\frac{{\beta _0} + {\beta _1}x_1 + {\beta
            _2}x_2}{{\gamma _0} + {\gamma _1}x_1 + {\gamma _2}x_2}
    \right),\,(\gamma_1,\gamma_2) \neq (0,0).$$

    Then: \begin{itemize}
        \item [(a)] The mapping $f$ is birational if
        and only if the vectors
        $(\beta_0,\beta_1,\beta_2),\,\,(\gamma_0,\gamma_1,\gamma_2)$ are
        linearly independent and
        $((\alpha\beta)_{12},(\alpha\gamma)_{12})\ne
        (0,0),\,((\alpha\gamma)_{12},(\beta\gamma)_{12})\ne (0,0),$ and
        either $((\alpha\beta)_{12},(\beta\gamma)_{12})\ne (0,0)$ or
        $(\beta_1,\,\beta_2) = (0,0).$

        \item [(b)] The mapping $f$ is degenerate if and only if $(\beta\gamma)_{12}=0$ or $(\alpha\gamma)_{12}=0.$
    \end{itemize}\end{lemma}

    \begin{proof}
        The conditions in {\it (a)} are necessary for $f$ to be invertible
        as if the vectors $(\beta_0,\beta_1,\beta_2),$
        $(\gamma_0,\gamma_1,\gamma_2)$ are linearly dependent then the
        second component of $f$ is a constant, also if
        $((\alpha\beta)_{12},$ $(\alpha\gamma)_{12})= (0,0)$ or
        $((\alpha\gamma)_{12},(\beta\gamma)_{12})= (0,0)$ then $f$ only
        depends on $\alpha_1\,x_1+\alpha_2\,x_2$ or on
        $\gamma_1x_1+\gamma_2x_2.$ If
        $((\alpha\beta)_{12},(\beta\gamma)_{12})= (0,0)$ and
        $(\beta_1,\,\beta_2) \neq (0,0)$ then $f$ only depends on
        $\beta_1x_1+\beta_2x_2$.

        Now assume that conditions $(a)$ are satisfied. Then the inverse of
        $f$ which formally is
        $$f^{-1}(x,y)=\left(\frac{-(\alpha\beta)_{02}+\beta_2 x+(\alpha\gamma)_{02}y-\gamma_2 x
            y}{(\alpha\beta)_{12}-(\alpha\gamma)_{12}y},\frac{(\alpha\beta)_{01}-\beta_1
            x+(\alpha\gamma)_{10}y+\gamma_1 x
            y}{(\alpha\beta)_{12}-(\alpha\gamma)_{12}y}\right),$$ is well
        defined. Furthermore the numerators of the determinants of the
        Jacobian of $f$ and $f^{-1}$ are
        \begin{equation}\label{detf}
        \alpha_1(\beta\gamma)_{02}-\alpha_2(\beta\gamma)_{01}+\alpha_1(\beta\gamma)_{12}x+\alpha_2(\beta\gamma)_{12}y
        \end{equation}
        and
        \begin{equation}\label{detfinv}
        \alpha_0(\beta\gamma)_{12}-\alpha_1(\beta\gamma)_{02}+\alpha_2(\beta\gamma)_{01}-(\beta\gamma)_{12}y,
        \end{equation}
        respectively. It is easily seen that conditions $(a)$ imply that
        both (\ref{detf}) and (\ref{detfinv}) are not identically zero.
        Hence, $f\circ f^{-1}=f^{-1}\circ f=id$ in $\C^2\setminus V,$ where
        $V$ is the algebraic curve determined by the common zeros of
        (\ref{detf}) and (\ref{detfinv}).

        To see {\it (b)} we know that since $S_i$ maps to $A_i,$
        this implies that the points $A_0,A_1,A_2$ are not all distinct.
        Since $A_0\ne A_1$ we have two possibilities: $A_0=A_2$ or
        $A_1=A_2.$ Condition $A_0=A_2$ writes as
        $(\beta\gamma)_{12}\,(\alpha\gamma)_{12}=0$ and
        $(\alpha\beta)_{12}(\beta\gamma)_{12}=0.$ From {\it (a)}, the vector
        $((\alpha\beta)_{12},(\alpha\gamma)_{12})\ne (0,0).$ Hence
        $(\beta\gamma)_{12}$ must be zero. In a similar way it is seen that
        $A_1=A_2$ if and only if $(\alpha\gamma)_{12}=0.$
    \end{proof}

It is easy to see that $F$ maps each $S_i$ to $A_i$ and that the
inverse of $F$ maps $T_i$ to $O_i$ for $i\in \{0,1,2\}.$ To specify
this behaviour we write $F:S_i\twoheadrightarrow A_i$ (also
$F^{-1}:T_i\twoheadrightarrow O_i$). It is known that the dynamical degree depends on the orbits of $A_0,A_1,A_2$ under the action of $F$ (see Proposition of section 2). Indeed, the key point is whether the iterates of $A_0, A_1, A_2$ coincide with any of the indeterminacy  points of $F.$ When we find such orbit of iterates of $F$ that ends at some indeterminacy point of $F$ we perform a series of blow up in order to remove the indeterminacy of $F$ in the new extended space.

For $X=\{\left((x,y),[u:v]\right)\in \C^2\times\pru\,:\,xv=yu\}$ and $p\in\C^2$ we let $(X,\pi)$ to be the blowing-up of $\C^2$ at the point $p.$ By translating $p$ at the origin, $\pi^{-1}p=\pi^{-1}(0,0)=\{\left((0,0),[u:v]\right)\}:=E_p\simeq\pru$ and $\pi^{-1}q=\pi^{-1}(x,y)=\left((x,y),[x:y]\right)\in X$ for $q=(x,y)\ne (0,0).$ Every blow up gives a new expanded space $X$
and a new induced map $\tilde{F}: X \to X$ is defined on it. Indeterminacy sets and exceptional locus can
also be defined by considering meromorphic functions on
complex manifolds $X$ we get after a series of blow ups. Consider the Picard group of $X$ denoted by $\mathcal{P}ic(X),$ where $X$ is the complex manifold. For a generic line $L \in \pr$ the $\mathcal{P}ic(\pr)$ is generated by the class of $L.$ If the base points of the blow-ups are $\{p_1, p_2,\ldots ,p_k\}\subset \pr$ and $E_i:=\pi^{-1}\{p_i\}$ then it is known that $\mathcal{P}ic(X)$ is generated by $\{\hat L, E_1, E_2, \ldots ,E_k\},$ see
\cite{BK2, BK3}. The curve $\hat L$ is the {\it strict transform} of $L \in \C^2$ is the adherence of ${\pi^{-1}}{(C\setminus\{p\})},$ in the Zariski topology. Furthermore $\pi:X\longrightarrow \pr$ induces a
morphism of groups $\pi^*:\mathcal{P}ic(\pr)\longrightarrow
\mathcal{P}ic(X),$ with the property that for any complex curve
$C\subset \pr,$
\begin{equation}\label{clau}
\pi^*(C)=\hat C+\sum m_i\,E_i,
\end{equation}
where $m_i$ is the algebraic multiplicity of $C$ at $p_i.$
For  $F \in \pr,$ $\tilde F$ is denoted as  natural extension of $F$ on $X$ and it induces a  morphism of groups,
$\tilde F^*:\mathcal{P}ic(X)\rightarrow \mathcal{P}ic(X)$ by
considering the classes of preimages such that
$\tilde F^*(\hat L)\,=\,d\,\hat L\,+\,\sum_{i=1}^k c_i\,E_i\quad ,
\quad c_i\in\Z,$
where $d$ is the degree of $F.$ By iterating $F,$ we get the
corresponding formula by changing $F$ by $F^n$ and $d$ by $d_n.$
To know the behavior of the sequence of degrees $d_n$ we deal with maps $\tilde F$ such that
\begin{equation}\label{AS}
(\tilde F^{n})^*=(\tilde F^*)^n.
\end{equation}
Maps $\tilde F$  satisfying condition (\ref{AS}) are called {\it
    Algebraically Stable maps} (AS for short), (see \cite{DF}). In order to get AS maps we will use the following useful result
showed by Fornaess and Sibony in \cite{FS} (see also Theorem 1.14)
of \cite{DF}:
\begin{equation}\label{condicio}
{\text{The map $\tilde{F}$ is AS if and only if for every
        exceptional curve}\,\, C\,\, \text{and all} \,\, n\ge 0\,,\,\tilde
    F^n(C)\notin {\mathcal{I}}(\tilde F).}
\end{equation}

It is known (see Theorem 0.1 of \cite{DF}) that one can always
arrange for a birational map to be AS performing a finite number of
blowing-up's. If it is the case and we call $\mathcal
{X}(x)=x^k+\sum_{i=0}^{k-1} c_i\,x^i$ the characteristic polynomial
of $A:=(\tilde F^*),$ then since $\mathcal {X}(A)=0$ and $d_i$ is
the $(1,1)$ term of $A^i$ we get that
$d_{k}=-(c_0+c_1d_1+c_2d_2+\cdots + c_{k-1}d_{k-1}),$
i.e., the sequence $d_n$ satisfies a homogeneous linear recurrence
with constant coefficients. The dynamical degree is then the largest
real root of $\mathcal {X}(x).$ The following is a direct consequence of Theorem $0.2$ of \cite{DF}. It is quiet useful in our work. Given a birational map $F$ of $\pr,$ let $\tilde{F}$ be its regularized map so that the induced map $\tilde{F}^{*}:\mathcal{P}ic(X) \to \mathcal{P}ic(X)$ satisfies $(\tilde{F}^n)^{*} = (\tilde{F}^{*})^n.$ Then

\begin{theorem}\label{theo-diller}
    Let $F:\pr \to \pr$ be a birational map and let $d_n = deg(F^n).$ Then up to bimeromorphic conjugacy, exactly one of the following holds:
    \begin{itemize}
        \item The sequence $d_n$ grows quadratically and $\tilde{F}$ is an automorphism preserving an elliptic fibration.
        \item The sequence $d_n$ grows linearly and $\tilde{F}$ preserves a rational fibration. In this case $\tilde{F}$ cannot be conjugated to an automorphism.
        \item The sequence $d_n$ is bounded and $\tilde{F}$ preserves a two generically transverse rational fibrations and $\tilde{F}$ is an automorphism.
        \item The sequence $d_n$ grows exponentially.
    \end{itemize}
    In the first three cases $\delta(F) = 1$ while in the last one $\delta(F) > 1.$ Furthermore in the first and second, the invariant fibrations are unique.
\end{theorem}

Since we only deal with degenerate maps, we have to consider two
subfamiles: $(\beta\gamma)_{12}=0$ or $(\alpha\gamma)_{12}=0.$ We
begin with the simplest case $(\alpha\gamma)_{12}=0.$

\section{Subfamily $(\alpha\gamma)_{12}=0.$}
\begin{lemma}\label{alfagama}
Consider  birational mappings $$f(x_1,x_2)=\left( {\alpha _0} +
{\alpha _1}x_1 + {\alpha _2}x_2,\frac{{\beta _0} + {\beta _1}x_1 +
{\beta _2}x_2}{{\gamma _0} + {\gamma _1}x_1 + {\gamma _2}x_2}
\right),\,(\gamma_1,\gamma_2) \neq (0,0)$$ with the condition
$(\alpha\gamma)_{12}=\alpha_1\gamma_2-\alpha_2\gamma_1=0.$ Then
either \begin{itemize}
\item [(i)]The four numbers $\alpha_1,\alpha_2,\gamma_1,\gamma_2$
are distinct for zero. \item[(ii)] $\alpha_1=0,\gamma_1=0$ and
$\alpha_2\ne 0\ne\gamma_2.$
\item[(iii)] $\alpha_2=0,\gamma_2=0$ and
$\alpha_1\ne 0\ne\gamma_1.$
\end{itemize}
\end{lemma}

\begin{proof}
From Lemma \ref{conditions} we know that $(\alpha_1,\alpha_2) \neq
(0,0).$ Then if $\alpha_1$ (resp. $\gamma_1$) is zero then
$\alpha_2$ (resp. $\gamma_2$) is not and from
$\alpha_1\gamma_2-\alpha_2\gamma_1=0$ we get that $\gamma_1$ (resp.
$\alpha_1$) must be zero.
\end{proof}

\begin{theorem}\label{CD2}

Consider  birational mappings $$f(x_1,x_2)=\left( {\alpha _0} +
{\alpha _1}x_1 + {\alpha _2}x_2,\frac{{\beta _0} + {\beta _1}x_1 +
{\beta _2}x_2}{{\gamma _0} + {\gamma _1}x_1 + {\gamma _2}x_2}
\right),\,(\gamma_1,\gamma_2) \neq (0,0)$$ with the condition
$(\alpha\gamma)_{12}=\alpha_1\gamma_2-\alpha_2\gamma_1=0.$ Then the
following hold:
\begin{itemize}
\item [(i)] If $\alpha_1\ne 0,\alpha_2\ne 0, \gamma_1\ne 0$ and $\gamma_2\ne
0,$ then $\delta(F)=2.$
\item [(ii)] If $\alpha_1=\gamma_1=0$, then $\delta(F)=\delta^{*}$ and $d_{n+2}=d_{n+1}+d_n.$

\item [(iii)] If $\alpha_2=\gamma_2=0$, then $\delta(F)=1$ and $d_n=1+n.$
\end{itemize}
\end{theorem}

\begin{proof}
From the hypothesis we have that:
${\mathcal{E}(F)}\,=\,\{S_0,S_1\}\,,\,{\mathcal{I}(F)}\,=\,\{O_0,O_1\},$
${\mathcal{E}(F^{-1})}\,=\,\{T_0,T_1\}$ and
${\mathcal{I}(F^{-1})}\,=\,\{A_0,A_1\}$ with
$$S_0=\{x_0=0\}\,,\,S_1=\{\gamma_0x_0+\gamma_1x_1+\gamma_2x_2=0\},$$
$$O_0=[(\beta\gamma)_{12}:(\beta\gamma)_{20}:(\beta\gamma)_{01})]\,,\,O_1=[0:\alpha_2:-\alpha_1],$$
$$T_0=\{(\beta_2(\alpha\gamma)_{01}-\beta_1(\alpha\gamma)_{12})x_0-(\beta\gamma)_{12}x_1=0\}\,,\,T_1=\,\{x_0=0\},$$
$$A_0=[0:1:0]\,,\,A_1=[0:0:1].$$
When $\alpha_1, \alpha_2, \gamma_1$ and $\gamma_2$ are non zero we
observe that $A_0\ne O_0$ and $A_0\ne O_1.$ Hence, since
$F(A_0)=[0:\alpha_1\,\gamma_1:0]=A_0$ and $F(A_1)=
[0:\alpha_2\,\gamma_2:0]=A_0$ we get that $F$ is AS. It implies that
$d_n=2^n$ and consequently $\delta(F)=2$.

To prove (ii) we observe that $\alpha_1=\gamma_1=0$ not only implies
that $(\alpha_2,\gamma _2)\ne (0,0)$ but also that $\beta_1\ne 0$
(if not $f$ would only depend on $y$ and it would not be
birational). Now $A_0=O_1\in\mathcal{I}(F)$ and we have to blow-up
this point. Let $E_0$ be the principal divisor at this point and
consider a point
 $[u:v]_{E_0}\in E_0.$
In order to extend $F$ on $E_0$ we see $[u:v]_{E_0}$ as $\lim_{t\to 0}[tu:1:tv]$ and  we are going to evaluate $F[tu:1:tv]:$
$$F[tu:1:tv]=[u(\gamma_0 u+\gamma_2 v)t:(\alpha_0 u+\alpha_2 v)(\gamma_0 u+\gamma_2 v)t:\beta_1 u+(\beta_0 u+\beta_2 v)ut].$$
Taking the limit when $t$ tends to zero we have that when $u\ne 0\,,\,\tilde{F}[u:v]_{E_0}=[0:0:1]$ while $[0:1]_{E_0}$ becomes an indeterminacy point
for $\tilde{F}.$

To know the action of $\tilde{F}$ on $S_0$ we see the point
$[0:x_1:x_2]$ as $\lim_{t\to 0}[t:x_1:x_2].$ Then for $t\to 0$ (and
$x_2\ne 0$)
$$\lim_{t\to 0}F[t:x_1:x_2]=\lim_{t\to 0}[\gamma_2x_2t:\alpha_2\gamma_2x_2^2:(\beta_1x_1+\beta_2x_2)t]= [\gamma_2x_2:\beta_1x_1+\beta_2x_2]_{E_0}.$$
The above considerations imply that
$\mathcal{I}(\tilde{F})=\{O_0,[0:1]_{E_0}\}\,,\,\mathcal{E}(\tilde{F})=\{\hat{S}_1,E_0\}$
with $\hat{S_1}\twoheadrightarrow A_1$ and $E_0\twoheadrightarrow
A_1.$ To follow the orbit of $A_1$ under $\tilde{F}$ we observe that
$A_1=[0:0:1]\in S_0$ and hence
$\tilde{F}[0:0:1]=[\gamma_2:\beta_2]_{E_0}\ne [0:1]_{E_0}$ and which
is sent to $A_1$ again giving a two-periodic orbit. It implies that
$\tilde{F}:X\longrightarrow X$ is AS. The Picard group of $X$ is
$Pic(X)=<\hat{L},E_0>$ where $L$ is a generic line of $\pr.$ Let
$\tilde{F}^*$ denote the corresponding map on  $Pic(X),$ which acts
just taking preimages. Hence $\tilde{F}^*(E_0)=\hat{S_0}.$ In order
to write $\hat{S_0}$ as a linear combination of $\hat{L},E_0$ we are
going to use (\ref{clau}). We have that
$\pi^*(S_0)=\hat{S_0}+E_0=\hat{L}$ which implies that
$\tilde{F}^*(E_0)=\hat{L}-E_0.$ Also
$\pi^*(F^{-1}(L))=\hat{F^{-1}}(L)+E_0=2\hat{L}$ which implies that
$\tilde{F}^*(\hat{L})=2\hat{L}-E_0.$ Hence the matrix of
$\tilde{F}^*$ on $Pic({X})=<\hat{L}, E_0>$  is
\begin{footnotesize}{$ \left(
                     \begin{array}{ccc}
                       2 & 1  \\
                       -1 & -1 \\

                     \end{array}
                   \right)$}\end{footnotesize}
with characteristic polynomial $z^2-z-1.$ Hence
$\delta(F)=\delta^{*}.$

To prove (iii) we observe, as before,  that $\alpha_2=\gamma_2=0$
not only implies that $(\alpha_1,\gamma _1)\ne (0,0)$ but also that
$\beta_2\ne 0$ (if not $f$ would only depend on $x$ and it would not
be birational). Now $A_1=O_1=[0:0:1]\in\mathcal{I}(F)$ and we have
to blow-up this point. To know the action of $\tilde{F}$ on $S_1$ we
see the point $[\gamma_1 x_0:-\gamma_0 x_0:\gamma_1 x_2]$ as
$\lim_{t\to 0}[\gamma_1 x_0:t-\gamma_0 x_0:\gamma_1 x_2].$ Similar
computations as before give us that each point in
$S_1\setminus\{O_0,[0:0:1]\}$ is sent to the point
$[\gamma_1:(\alpha\gamma)_{01}]_{E_1}.$ That is $\hat{S_1}$  is
still exceptional for $\tilde{F}.$

Now consider a point $[u:v]_{E_1}\in E_1.$ It is seen as $\lim_{t\to
0}[tu:tv:1].$ Then for $t\to 0,$
$$\lim_{t\to 0}F[tu:tv:1]= \lim_{t\to 0}[tu(\gamma_0 u+\gamma_1 v):t(\gamma_0 u+\gamma_1 v)(\alpha_0 u+\alpha_1 v):\beta_2 u].$$

If $\gamma_0 u+\gamma_1 v\ne 0$ and $u\ne 0$ then
$\tilde{F}[u:v]_{E_1}=[u:\alpha_0 u+\alpha_1 v]_{E_1}.$

If $\gamma_0 u+\gamma_1 v=0,$ in the above computation with
$[u:v]_{E_1}=[\gamma_1:-\gamma_0]_{E_1}$ we deal with the point
$[\gamma_1 t:-\gamma_0 t:1]\in S_1$ and we have to apply $\tilde{F}$
giving $\tilde{F}[\gamma_1 t:-\gamma_0
t:1]=[\gamma_1:(\alpha\gamma)_{01}]_{E_1}.$ We observe that
$\lim_{u\to \gamma_1,v\to -\gamma_0}
 \tilde{F}[u:v]_{E_1}=[\gamma_1:(\alpha\gamma)_{01}]_{E_1},$ that is
$\tilde{F}$ is well defined.

But if $u=0\,,\,F[0:t:1]=[0:1:0].$ It implies that $[0:1]_{E_1}\in
{\mathcal{I}}(\tilde{F}).$

We claim that after this blow-up the map $\tilde{F}$ is AS. It is so
because $S_0\twoheadrightarrow A_0$ and $A_0$ is a fixed point of
$F$ and $\hat{S_1}\twoheadrightarrow
[\gamma_1:(\alpha\gamma)_{01}]_{E_1}$ and the iterates of this point
 never coincide with $[0:1]_{E_1}.$ The Picard group of $X$ is
now $Pic(X)=<\hat{L},E_1>$ where $L$ is a generic line of $\pr,$
$\tilde{F}^*(E_1)=\hat{S_1}+E_1,$ and similar computations as the
ones of (ii) give the matrix
{$$ \left(
                     \begin{array}{ccc}
                       2 & 1  \\
                       -1 &  0\\

                     \end{array}
                   \right).$$}

The characteristic polynomial is given by $(z-1)^2.$ Hence
$\delta(F)=1.$ Furthermore, since $d_1=2$ and $d_2=3$ we get that
$d_n=1+n.$
\end{proof}

\begin{proposition}\label{zeroentropyone}
Assume that
$$f(x_1,x_2)=\left( {\alpha _0} +
{\alpha _1}x_1 + {\alpha _2}x_2,\frac{{\beta _0} + {\beta _1}x_1 +
{\beta _2}x_2}{{\gamma _0} + {\gamma _1}x_1 + {\gamma _2}x_2}
\right),\,(\gamma_1,\gamma_2) \neq (0,0)$$ with the condition
$(\alpha\gamma)_{12}=\alpha_1\gamma_2-\alpha_2\gamma_1=0,$ has zero
entropy. Then after an affine change of coordinates it can be
written as
$$f(x,y) = \left(\alpha_0 + \alpha_1\,x,\frac{\beta_0 + y}{x}\right),\quad \alpha_1 \ne 0.$$
    This map preserves the fibration $V(x,y) = x$ and this fibration is unique.

If $m(x):=\alpha_0+\alpha_1 x$ is periodic of period $p,$  that is
if $\alpha_1^p=1$ for some $p>1\,,\,\alpha_1\ne 1,$ then
$$W(x,y)=x\cdot m(x)\cdot m(m(x))\cdots  m^{p-1}(x)$$
is a first integral of $f(x,y).$
    Also when $\alpha_1=1$ and $\alpha_0=0,$ $f$ is integrable.
\end{proposition}

\begin{proof}
From Theorem \ref{CD2} we know that the only zero entropy maps in
the family are the ones with $\alpha_2=\gamma_2=0$ and we also know
that in this case $\beta_2,\alpha_1$ and $\gamma_1$ are different
from zero. Hence we can conjugate $f(x,y)$ with
$h(x,y)=\left(\frac{\beta_2}{\gamma_1}x-\frac{\gamma_0}{\gamma_1},\frac{1}{\beta_2}y+\frac{\beta_1}{\gamma_1}\right).$
Renaming the parameters we see that the conjugate map is of the form
$$f(x,y) = \left(\alpha_0 + \alpha_1\,x,\frac{\beta_0 + y}{x}\right),\quad \alpha_1 \ne 0.$$
Clearly this map preserves the fibration $V(x,y)=x$ and this
fibration is unique from Theorem \ref{theo-diller}. If
$\alpha_1^p=1$ for some $p>1\,,\,\alpha_1\ne 1,$ then
$W(f(x,y))=W(x,y)$ and the result follows. When $\alpha_1=1$ then we
see that $f(x,y)$ is integrable if and only if $\beta_0=0.$
\end{proof}
Now we are going to consider the second subfamily.

\section{Subfamily ${\mathbf{({\mathbf\beta}{\mathbf\gamma})_{12}=0}}.$}

 We are going to consider
three different cases, depending on $\gamma_1\gamma_2\ne
0\,,\,\gamma_1=0$ and $\gamma_2=0.$ When
$(\beta\gamma)_{12}=\beta_1\,\gamma_2-\beta_2\,\gamma_1=0,$ we have
that
${\mathcal{E}(F)}\,=\,\{S_0,S_1\}\,,\,{\mathcal{I}(F)}\,=\,\{O_0,O_1\},$
${\mathcal{E}(F^{-1})}\,=\,\{T_0,T_1\}$ and
${\mathcal{I}(F^{-1})}\,=\,\{A_0,A_1\}$ with
$$S_0=\{x_0=0\}\,,\,S_1=\{\gamma_0x_0+\gamma_1x_1+\gamma_2x_2=0\}$$
$$O_0=[0:\gamma_2:-\gamma_1]\,,\,O_1=[0:\alpha_2:-\alpha_1]$$
$$T_0=\{x_0=0\}\,,\,T_1=\{(\alpha\beta)_{12}x_0-(\alpha\gamma)_{12}x_2=0\}$$
$$A_0=[0:1:0]\,,\,A_1=[0:0:1].$$

\subsection {$(\beta\gamma)_{12}=0$ with $\gamma_1\gamma_2\ne 0$}
\begin{theorem}\label{CD3}
Consider  birational mappings $$f(x_1,x_2)=\left( {\alpha _0} +
{\alpha _1}x_1 + {\alpha _2}x_2,\frac{{\beta _0} + {\beta _1}x_1 +
{\beta _2}x_2}{{\gamma _0} + {\gamma _1}x_1 + {\gamma _2}x_2}
\right),\,(\gamma_1,\gamma_2) \neq (0,0)$$ with the conditions
$(\beta\gamma)_{12}=0$ and $\gamma_1\gamma_2\ne 0.$ Then either
\begin{itemize}
\item [(i)] $\alpha_1\ne 0\ne \alpha_2$ and $\delta(F)=2$ with
$d_n=2^n$ for all $n \in \N.$

\item[(ii)] $\alpha_1=0$ and the dynamical degree is
        $\delta(F)=\delta^*$ with $d_{n+2}=d_{n+1} + d_n$ for all $n \in \N.$

\item[(iii)] $\alpha_2=0$ and
the dynamical degree is $\delta(F)=1$ with $d_n=1+n$ for all
        $n \in \N.$
\end{itemize}
\end{theorem}

\begin{proof}

To prove $(i)$ we observe that $S_0 \twoheadrightarrow A_0$ and $S_1
\twoheadrightarrow A_1$ with $F(A_0) =
[0:\alpha_1\gamma_1:0]=A_0\notin \mathcal{I}(F)$ and $F(A_1)
        =[0:\alpha_2\gamma_2:0] =A_0\notin \mathcal{I}(F).$ Thus using
        (\ref{condicio}) we see that $F$ is AS, which implies that $d_n=2^n$
        and consequently $\delta(F)=2$.

Now consider that $\alpha_1=0.$ It implies that $\alpha_2\ne 0.$ In
this case $S_0 \twoheadrightarrow A_0  = O_1 \in \mathcal{I}(F).$
Hence we blow-up $A_0$ to obtain $E_0$. Similar computations as
before says that $\tilde{F}$ sends ${\hat{S_0}} \to E_0 \to
\hat{T_1}$ and no new indeterminacy points are created.

Now we have to follow the orbit of $A_1$ under the action of
$\tilde{F}.$ As $A_1 \in S_0$ we find that $\tilde{F}(A_1) =
[\gamma_2:\beta_2]_{E_0}$ and
$\tilde{F}[\gamma_2:\beta_2]_{E_0}=[\gamma_1\gamma_2:\alpha_0\gamma_2+\alpha_2\beta_2:\beta_1\gamma_2]\in
T_1.$ Observe that $\mathcal{I}(\tilde{F})=\{O_0\}$ and $O_0 \in S_0
= T_0.$ We know that the only points on $T_0$ which have preimages
are $A_0$ and $A_1$ which implies that if the iterates of $A_1$
reaches $O_0$ for some iterate of $F$ then $O_0$ should be equal to
either $A_0$ or $A_1.$ But the conditions on the parameters implies
that $A_0 \neq O_0 \neq A_1.$ This implies that $O_0$ has no
preimages hence the iterates of $A_1$ cannot reach $O_0.$ Hence we
see that $\tilde{F}$ is AS.

In this case $\tilde{F}^*(\hat{L})=2\hat{L}-E_0$ and
$\tilde{F}^*(E_0)=\hat{L}-E_0.$ Hence the characteristic polynomial
of the corresponding matrix is $z^2-z-1.$ It implies that the
dynamical degree is $\delta(F) = \frac{1+\sqrt{5}}{2}$ and $d_{n+2}
=d_{n+1}+d_n$ for all $n \in \N.$

Finally to see $(iii),$ since $\alpha_2=0$ we get that $\alpha_1\ne
0.$ Now we observe that $S_0$ collapses to $A_0=[0:1:0]\in S_0$ and
that $F[0:1:0]=[0:\alpha_1\gamma_1:0]=[0:1:0].$ Hence $A_0$ is a
fixed point.

The other exceptional curve $S_1 \twoheadrightarrow A_1= O_1 =
[0:0:\alpha_1]= [0:0:1]\in \mathcal{I}(F).$ Hence we have to blow-up $A_1$
obtaining  $E_1$. Similar computations as before says that $\tilde{F}$
sends ${\hat{S_1}} \to E_1 \to \hat{T_1}$ and no new indeterminacy
points are created. After this blow-up the mapping $\tilde{F}$ is
AS. And we can see that $\tilde{F}^*(\hat{L})=2\hat{L}-E_0$ and
$\tilde{F}^*(E_1)=\hat{L}.$ Hence the matrix of $\widetilde{F}^ *$ is:
\begin{equation}\left(
        \begin{array}{ccc}
        2 & 1 \\
        -1 & 0
        \end{array}
        \right).
        \end{equation}
The characteristic polynomial  is $(z-1)^2,$
and hence the dynamical degree is $1.$ Since $d_1 = 2,\,d_2 = 3$ we
get that the sequence of degrees is $d_n= 1 + n$ for all $n \in \N.$
\end{proof}
Concerning the zero entropy we see that the only case is the third
one, when $\alpha_2=0.$ The result (and the proof) we get is very
similar to the one stated in Proposition \ref{zeroentropyone}.

\begin{proposition}\label{zeroentropytwo}
Let $$f(x_1,x_2)=\left( {\alpha _0} + {\alpha _1}x_1 + {\alpha
_2}x_2,\frac{{\beta _0} + {\beta _1}x_1 + {\beta _2}x_2}{{\gamma _0}
+ {\gamma _1}x_1 + {\gamma _2}x_2} \right),\,(\gamma_1,\gamma_2)
\neq (0,0)$$ with the conditions $(\beta\gamma)_{12}=0$ and
$\gamma_1\gamma_2\ne 0$ and assume that $f(x,y)$ has zero entropy.
 Then after an affine change of coordinates it can be
written as
$$f(x,y) = \left(\alpha_0 + \alpha_1\,x,\frac{\beta_0}{x+y}\right),\quad \alpha_1 \ne 0.$$
    This map preserves the fibration $V(x,y) = x$ and this fibration is unique.
If $m(x):=\alpha_0+\alpha_1 x$ is periodic of period $p,$  that is
if $\alpha_1^p=1$ for some $p>1\,,\,\alpha_1\ne 1,$ then
$$W(x,y)=x\cdot m(x)\cdot m(m(x))\cdots  m^{p-1}(x)$$
is a first integral of $f(x,y).$
    Also when $\alpha_1=1$ and $\alpha_0=0,$ $f$ is integrable.

\end{proposition}

\subsection  {$(\beta\gamma)_{12}=0$ with $\gamma_1=0.$}
Next Theorem gives the behaviour of $d_n$ in this family. As we can
see below after affine change of coordinates these mappings are
simple, and the sequence of degrees can be deduced by elementary
methods. We have adopted this point of view in the proof of item
$(ii).$ But in the first part we have preferred the blow-up
approach. In fact some multiple blow-up's are implemented and it is
amazing to see how the method detects the different behaviours of
$d_n.$

\begin{theorem}\label{gamma1}
Consider  birational mappings $$f(x_1,x_2)=\left( {\alpha _0} +
{\alpha _1}x_1 + {\alpha _2}x_2,\frac{{\beta _0} + {\beta _1}x_1 +
{\beta _2}x_2}{{\gamma _0} + {\gamma _1}x_1 + {\gamma _2}x_2}
\right),\,(\gamma_1,\gamma_2) \neq (0,0)$$ with the conditions
$(\beta\gamma)_{12}=0$ and $\gamma_1= 0.$
\begin{itemize}
\item [(i)] Assume that $\alpha_2\ne 0.$ Then, after an affine change of coordinates $f(x,y)$ can be written as
\begin{equation}\label{EQQ1}
f(x,y)=\left(\alpha_0+\alpha_1
x+y,\frac{\beta_0}{\gamma_0+y}\right)\,,\,\alpha_1\ne 0\ne \beta_0
\end{equation}
and the following hold:
\begin{itemize}
\item[(a)] If the onedimensional mapping $h(y):=\frac{\beta_0}{\gamma_0+y}$ is not a periodic map, then the sequence of degrees is $d_n=1+n.$
\item[ (b)] If $h(y)$ is a k- periodic map and $1+\alpha_1^k+\alpha_1^{2k}+\cdots +\alpha_1^{nk}\ne 0$ for all $n\in\N,$ then $d_n=1+n$ for all $n\le k-1$
and $d_n=k$ for all $n\ge k.$
\item [(c)]  If $h(y)$ is a k-periodic map and $1+\alpha_1^k+\alpha_1^{2k}+\cdots +\alpha_1^{nk}= 0$ for some $n\in\N,$ then $d_n$ is a $(n+1)k-$periodic sequence.
\end{itemize}
\item[(ii)] Assume that $\alpha_2=0.$ Then, after an affine change of coordinates $f(x,y)$ can be written as
\begin{equation}\label{EQQ2}
f(x,y)=\left(\alpha_0+\alpha_1
x,\frac{\beta_0}{\gamma_0+y}\right)\,\,,\,\,\alpha_1\ne 0\ne \beta_0
\end{equation}
and the following hold:
\begin{itemize}
\item[(a)] If the onedimensional mapping $h(y):=\frac{\beta_0}{\gamma_0+y}$ is not a periodic map, then
$d_n=2$ for all $n\in\N.$
\item[ (b)] If $h(y)$ is a k- periodic map then $d_n$ is a
$k-$periodic sequence.
\end{itemize}
\end{itemize}
\end{theorem}

\begin{proof}
We notice that since $\gamma_1=0\,,\,\gamma_2\ne 0$ we can conjugate
$f(x,y)$ with
$$\psi(x,y)=\left(\frac{\alpha_2}{\gamma_2}x,\frac{1}{\gamma_2}y+\frac{\beta_2}{\gamma_2}\right)$$
and renaiming the coefficients if necessary, we get the desired map
(\ref{EQQ1}). Now we have that
$$S_0=\{x_0=0\}\,,\,S_1=\{\gamma_0 x_0+x_2=0\}\,,\,A_0=[0:1:0]\,,\,A_1=[0:0:1],$$
and
$$T_0=\{x_0=0\}\,,\,T_1=\{x_2=0\}\,,\,O_0=[0:1:0]\,,\,O_1=[0:1:-\alpha_1].$$
Since $A_0=O_0$ we have to blow-up this point getting $E_0.$ Then
$$\tilde{F}[u:v]_{E_0}=[\gamma_0u+v:\beta_0 u]_{E_0}\,\,,\,\,[u:v]_{E_0}\ne [1:-\gamma_0]_{E_0}$$
and
$$\hat{S}_0\twoheadrightarrow [1:0]_{E_0}.$$
The point $[1:-\gamma_0]_{E_0}$ is now an indeterminacy point of $\tilde{F}.$ Hence, if $\tilde{F}^p[1:0]_{E_0}\ne [1:-\gamma_0]_{E_0}$ for all $p\in\N,$ since $\hat{S}_1\twoheadrightarrow A_1\in S_0\,\,,\,\, \tilde{F}(A_1)= [1:0]_{E_0}$ and we get that $\tilde{F}$ is AS. It can be seen that the matrix of $\tilde{F}^*:Pic(X)\rightarrow Pic(X)=<\hat{L},E_0>$ is
\begin{equation}\left(
        \begin{array}{ccc}
        2 & 1 \\
        -1 & 0
        \end{array}
        \right).
        \end{equation}
The characteristic polynomial  is $(z-1)^2,$
and hence the dynamical degree is $1.$ Since $d_1 = 2,\,d_2 = 3$ we
get that the sequence of degrees is $d_n= 1 + n$ for all $n \in \N.$

Now assume that there exists some $p\in\N$ such that
$\tilde{F}^p[1:0]_{E_0}=[1:-\gamma_0]_{E_0}.$ We claim that in this
case $\tilde{F}:E_0\rightarrow E_0$ is a $(p+2)-$periodic map. To
prove the claim we distinguish between the case $\gamma_0=0$ (which
gives a 2-periodic map and corresponds to $p=0$) and the case
$\gamma_0\ne 0.$ We have that
$$[1:0]_{E_0}\longrightarrow^{\tilde{F}^p}   [1:-\gamma_0]_{E_0}\longrightarrow^{\tilde
{F}}[0:1]_{E_0}\longrightarrow^{\tilde {F}}[1:0]_{E_0}. $$ Hence
$\tilde{F}^{p+2}$ which  in fact is a Moebius map, fixes at least
three different points. It clearly implies that  $\tilde{F}^{p+2}$
is the identity map. Since the restriction of $\tilde{F}$ at $E_0$
is exactly the map $h(y)=\frac{\beta_0}{\gamma_0+y}$  extended to
the projective line we can assert that
$\tilde{F}^p[1:0]_{E_0}=[1:-\gamma_0]_{E_0}$ if ad only of $h(y)$ is
a $(p+2)-$periodic map. Hence $(a)$ is proved.

Following the process, if
$\tilde{F}^p[1:0]_{E_0}=[1:-\gamma_0]_{E_0},$  we have to blow-up
all the points $\tilde{F}^j [1:0]_{E_0}$ for $j=0,1,\ldots,p.$ We
call $E_{0j}$ the corresponding principal divisors, getting:
\begin{equation} \label{seq}
E_{00}\longrightarrow E_{01}\longrightarrow E_{02}\longrightarrow
\cdots \longrightarrow E_{0p}.\end{equation} Calling again
$\tilde{F}$ the map at this new variety we are going to see which is
the image of $S_0$ and which is the image of $E_{0p}.$

A point of coordinate $k$ in $E_{00}$ is seen as $\lim_{t\to 0}[t:1:kt^2].$ Then, for any point in $S_0$ different from the indeterminacy points and for
$t\sim 0,$ we have that:
$$F(t,x_1,x_2)\sim [x_2 t:(\alpha_1 x_1+\alpha_2 x_2)x_2:\beta_0 t^2]=\left[\frac{t}{\alpha_1 x_1+\alpha_2 x_2}:1:\frac{\beta_0}{x_2(\alpha_1 x_1+\alpha_2 x_2)}t^2\right].$$
 Naming $T:=\frac {t}{\alpha_1 x_1+\alpha_2 x_2}$ this point looks like $\left[T:1:\frac{\beta_0(\alpha_1 x_1+\alpha_2 x_2)}{x_2}T^2\right],$ that is
$$\tilde{F}[0:x_1:x_2]=\frac{\beta_0(\alpha_1 x_1+\alpha_2 x_2)}{x_2}\in E_{00}.$$

Now consider a point of coordinate $k$ in $E_{0p}.$ This point is seen as $\lim_{t\to 0}[t:1:-\gamma_0 t+kt^2].$ Then for $t\sim 0,$
$$F[t:1:-\gamma_0 t+kt^2]\sim [kt:\alpha_1 k:\beta_0]\rightarrow_{t\to 0} [0:\alpha_1 k:\beta_0]\in S_0.$$
Hence (\ref{seq}) can be completed and we get the cycle:
$$\hat{S}_0\longrightarrow E_{00}\longrightarrow E_{01}\longrightarrow E_{02}\longrightarrow \cdots \longrightarrow E_{0p}\longrightarrow \hat{S}_0 .$$
Now since $S_1\twoheadrightarrow A_1\in S_0$ and $\tilde{F}^{p+2}$
sends $\hat{S}_0$ to itself,  it could happen that for some
$n\in\N\,,\,\tilde{F}^{n(p+2)}(A_1)=O_0,$ which still is an
indeterminacy point of $\tilde{F}.$

If it is not the case, these $\tilde{F}$ is AS. Let us to compute
the matrix of $\tilde{F}^*.$ The Picard group of $X$ is
$Pic(X)=<\hat{L},E_{00},E_{01},\ldots ,E_{0p},E_0>.$ To write
$\hat{S}_0$ and $\hat{S}_1$ as a linear combination of basis
elements, we are going to use the identity (\ref{clau}). For
instance $\pi^*(F^{-1}(L))=\hat{F^{-1}(L)}+\sum_{j=1}^p m_jE_{0j}$
where the multiplicities $m_j$ are the order of vanishing of
$F^{-1}(L)$ at generic points of $E_{0j}.$ If
$\delta_0x_0+\delta_1x_1+\delta_2x_2=0$ is the equation of a generic
straight line $L,$ then a calculation gives $\delta_0
F[t:1:wt+kt^2][1]+\delta_1 F[t:1:wt+kt^2][2]+\delta_2
F[t:1:wt+kt^2][3]=\delta_1\alpha_1(\gamma_0+w)t+o(t^2)$ which let us
to write $\pi^*(F^{-1}(L))=\hat{F^{-1}(L)}+\sum_{j=1}^{p-1} E_{0j}+2
E_{0p}.$ Now from $\pi^*(F^{-1}(L))=2\hat{L}$ we get that
$\tilde{F}^*(\hat{L})=2\hat{L}-\sum_{j=1}^{p-1} E_{0j}-2 E_{0p}.$
Proceeding in this way we find that the matrix of $\tilde{F}^*$ is:

$$\left (
\begin{array}{ccccccc}
2&1&0&0&\hdots &0&0\\
-1&-1&1&0&\hdots &0&0\\
0&0&0&1&\hdots &0&0\\
\vdots&\vdots&\vdots&\vdots&\ddots &\vdots&\vdots\\
-1&-1&0&0&\hdots &1&0\\
-2&-1&0&0&\hdots &0&0\\
-1&-1&0&0&\hdots &0&1\\
\end{array}
\right ). $$

It is not hard to see that the characteristic polynomial of such a
matrix is $(-z)^p\,(z-1)^2.$ Hence the sequence of degrees satisfies
$d_{n+p+2}=2d_{n+p+1}-d_{n+p}$ and its behaviour depends on the
initial conditions, i.e., on the first terms
$d_1,d_2,\ldots,d_{p+2}.$ So, if $h(y)$ is $k-$periodic, then
$k=p+2$ and $f^k(x,y)[2]=y.$ It implies that $d_k=d_{k-1}.$ Since
the first degrees are $2,3,4,\ldots,k,k,$ from
$d_{n+k}=2d_{n+k-1}-d_{n+k-2}$ we get that $d_n=k$ for all $n\ge k.$
It remains to prove that the condition $\tilde{F}^{n(p+2)}(A_1)=O_0$
is equivalent to $1+\alpha_1^k+\alpha_1^{2k}+\cdots +\alpha_1^{nk}=
0$ for $k=p+2.$ To this end, taking into account the terms of
maximum degree of $f^k(x,y),$ (see (\ref{condk}) below) we get that:
$$\tilde{F}^k[0:x_1:x_2]=[0:\alpha_1^{k-1}(\alpha_1x_1+x_2):x_2]$$
and hence
$$\tilde{F}^{nk}[0:0:1]=[0:\alpha_1^{k-1}(1+\alpha_1^k+\alpha_1^{2k}+\cdots
+\alpha_1^{(n-1)k}:1].$$ Therefore
$\tilde{F}^{nk}[0:0:1]=O_1=[0:1:-\alpha_1]$ if and only if
\begin{equation}\label{condnk}
1+\alpha_1^k+\alpha_1^{2k}+\cdots +\alpha_1^{nk}=0.\end{equation}
Statement $(b)$ is now proved. To see $(c)$ we just compute
$f^{(n+1)k}.$ In this case since $\alpha_1\ne 1$ we can consider
(doing a translation if necessary) that $\alpha_0=0.$ Now the
expression of $f^k$ is
\begin{equation}\label{condk}
f^k(x,y)=\left(\alpha_1^k x+\alpha_1^{k-1} y + \alpha_1^{k-2}
h(y)+\alpha_1^{k-3}
h^2(y)+\cdots+\alpha_1h^{k-2}(y)+h^{k-1}(y),y\right)\end{equation}
Hence:
$$f^{(n+1)k}(x,y)=\left(\alpha_1^{(n+1)k} x+(1+\alpha_1^k+\alpha_1^{2k}+\cdots +\alpha_1^{nk})(\alpha_1^{k-1} y + \alpha_1^{k-2}
h(y)+h^{k-1}(y)),y \right).$$ Then since condition (\ref{condnk})
implies that $\alpha_1^{(n+1)k}=1$ we have that when condition
(\ref{condnk}) is satisfied, $f$ is a $(n+1)k-$periodic map, and
hence also the sequence of degrees is $(n+1)k-$periodic.

We are going to prove $(ii).$ First of all,  $\gamma_1=0$ implies
$\gamma_2\ne 0,$ and from $\gamma_1\beta_2-\gamma_2\beta_1=0$ we get
$\beta_1=0.$ Doing a translation on $y$ and renaiming the
coefficients we get equation (\ref{EQQ2}). Since this map is very
simple we are going to prove the result on the behaviour of $d_n$
using simple arguments. We observe that the first component of
$f^k(x,y)$ is $a_k x+b_k$ for certains $a_k,b_k.$ And the second
components are just the iterates of
$h(y)=\frac{\beta_0}{\gamma_0+y},$ a one-dimensional M{\"o}bius map.
We claim that if $h(y)$ is not a periodic map then $h^k(y)$ is a
M{\"o}bius map with non-constant denominator for all $k\in\N$ and
also that the denominators of $h^i(y)$ and $h^j(y)$ are different
for $i\ne j.$ From the claim we can deduce that when $h(y)$ is not a
periodic map then $d_n=2$ for all $n\in\N.$ And when $h(y)$ is a
$k-$periodic map, then the sequence of degrees is $d_n=2$ for all
$n$ which is not a multiple of $k$ and $d_{n}=1$ when $n$ is a
multiple of $k.$

To prove the claim we consider $N_k$ and $D_k$ with
        $h^k(z)=\frac{N_k}{D_k}$ and we see that, if we don't perform
        simplifications, $N_{k+1}=\beta_0\,D_k$ and
        $D_{k+1}=\gamma_0\,D_k+N_k.$ Let $p_k,q_k\in\C$ such that
        $D_k=p_k+q_k\,z.$ Then,
        $D_{k+2}-\gamma_0\,D_{k+1}-\beta_0\,D_k=0,$ which implies
        that $q_{k+2}-\gamma_0\,q_{k+1}-\beta_0\,q_k=0.$ Analyzing
        this linear recurrence with constant coefficients and taking
        into account that this sequence is $k-$periodic if and only
        if $\left(\frac{\lambda_2}{\lambda_1}\right)^k=1,$ where
        $\lambda_1,\lambda_2$ are the two different roots of $\lambda^2-\gamma_0\,\lambda-\beta_0=0$
        (see \cite{CGM1}), the claim follows.

\end{proof}

\begin{proposition}\label{zeroentropythreee}
Consider the birational mappings

\begin{equation}\label{zeroentropythree}
f(x,y)=\left(\alpha_0+\alpha_1
x+y,\frac{\beta_0}{\gamma_0+y}\right)\,\,,\,\,\alpha_1\ne 0\ne
\beta_0.
\end{equation}
Then the following hold:
\begin{itemize}
\item[(a)] If the onedimensional mapping $h(y):=\frac{\beta_0}{\gamma_0+y}$ is not a periodic map, then
$f(x,y)$ has the unique invariant fibration $V_1(x,y)=y.$
\item[ (b)] If $h(y)$ is a k- periodic map and $1+\alpha_1^k+\alpha_1^{2k}+\cdots +\alpha_1^{nk}\ne 0$ for all $n\in\N,$
then  $f(x,y)$ is integrable being
$$H_1(x,y)=y+h(y)+h(h(y))+\cdots +h^{k-1}(y)$$
a first integral and also has a second invariant fibration
$V_2(x,y):$
\begin{itemize}
\item[ ($b_1$)] If $\alpha_1^k\ne 1$ we can assume that $\alpha_0=0$ and
then $V_2(x,y)=$
\begin{equation}\label{periodic}(\alpha_1^k-1) x+\alpha_1^{k-1} y + \alpha_1^{k-2}
h(y)+\alpha_1^{k-3}
h^2(y)+\cdots+\alpha_1h^{k-2}(y)+h^{k-1}(y)\end{equation} satisfies
$V_2(f(x,y))=\alpha_1V_2(x,y).$
\item[ ($b_2$)] If $\alpha_1^k=1$ but $\alpha_1\ne 1$ we can assume that $\alpha_0=0$ and
then $V_2(x,y)=$
\begin{footnotesize}$$\frac{kx+(k-1)\alpha_1^{k-1}y+(k-2)\alpha_1^{k-2}h(y)+(k-3)\alpha_1^{k-3}h(h(y))+\cdots +2\alpha_1^{2}h^{k-3}(y)+\alpha_1h^{k-2}(y)}
{\alpha_1^{k-1}y+\alpha_1^{k-2}h(y)+\alpha_1^{k-3}h(h(y))+\cdots
+\alpha_1h^{k-2}(y)+h^{k-1}(y)}$$\end{footnotesize}satisfies
$V_2(f(x,y))=V_2(x,y)+1.$
\item[ ($b_3$)] If $\alpha_1=1$ then
\begin{footnotesize}$$V_2(x,y)=\frac{kx+(k-1)y+(k-2)h(y)+(k-3)h(h(y))+\cdots +2h^{k-3}(y)+(y)}
{k\alpha_0+\alpha_1^{k-1}y+\alpha_1^{k-2}h(y)+\alpha_1^{k-3}h(h(y))+\cdots
+h^{k-2}(y)+h^{k-1}(y)}$$\end{footnotesize}satisfies
$V_2(f(x,y))=V_2(x,y)+1.$
\end{itemize}

\item [(c)]  If $h(y)$ is a k-periodic map and $1+\alpha_1^k+\alpha_1^{2k}+\cdots +\alpha_1^{nk}= 0$ for some $n\in\N,$ then
$f(x,y)$ has a second first integral $H_2(x,y)$ which can be given
by $H_2(x,y))=V_2^{(n+1)k}(x,y)$ being $V_2(x,y)$ be defined by
(\ref{periodic}).
\end{itemize}
\end{proposition}

The proofs are straightforward. Only say that to find the fibrations
we have considered combinations of $x,y,h(y),h(h(y)),\ldots,
h^{k-1}(y)$ or quotients of them.

\begin{remark}
Assuming the hypothesis $(b),$ since $d_n$ is a bounded sequence and
$f(x,y)$ is not a periodic map, from \cite{BD} we know that $f(x,y)$
is birationally equivalent to either $(x,y)\to (ax,by)$ where $a$ is
a root of unity and $b$ it is not or to $(x,y)\to (ax,y+1).$ The
fibrations encountered in $(b)$ let us to construct such a
conjugations. In fact, when $V_2(f(x,y))=\alpha_1V_2(x,y)$ we are in
the first case while when $V_2(f(x,y))=V_2(x,y)+1$ we are in the
second one.
\end{remark}

The invariant fibrations and first integrals corresponding to the
mappings satisfying $(ii)$ of Theorem \ref{gamma1} are very easy
after the adequate affine change of coordinates. Next Proposition
gives this information.

\begin{proposition}\label{zeroentropyfourr}
Consider the birational mappings
\begin{equation}\label{zeroentropyfour}
f(x,y)=\left(\alpha_0+\alpha_1
x,\frac{\beta_0}{\gamma_0+y}\right)\,\,,\,\,\alpha_1\ne 0\ne
\beta_0.
\end{equation}
These mappings preserve the two generically transverse invariant
foliations $V_1(x,y)=x$ and $V_2(x,y)=y.$ Furthermore,
\begin{itemize}
\item[(a)] If $h(y)=\frac{\beta_0}{\gamma_0+y}$ is periodic
of period $k$ then
$$H_1(x,y)=y+h(y)+h(h(y))+\cdots +h^{k-1}(y)$$
is a first integral of $f(x,y).$
\item[(b)] If $m(x):=\alpha_0+\alpha_1 x$ is periodic
of period $p$ then
$$H_2(x,y)=x+m(x)+m(m(x))+\cdots +m^{p-1}(x)$$
is a first integral of $f(x,y).$
\item[(c)] If $h(y)$ and $m(x)$ are $k-$periodic then $f(x,y)$ is a
$k-$periodic mapping having two independent first integrals
$H_1(x,y)$ and $H_2(x,y)$ with $p=k.$

\end{itemize}
\end{proposition}

\subsection {$(\beta\gamma)_{12}=0$ with $\gamma_2=0$}

If $\gamma_2=0$ we know that $\gamma_1\ne 0$ and from $(\beta\gamma)_{12}=0$ we get $\beta_2=0.$ Also $\alpha_2\ne 0,$ if not $f(x,y)$ only depens on $x.$

\begin{theorem}\label{gamma2}
Consider  birational mappings

\begin{equation}\label{t12}
f(x_1,x_2)=\left( {\alpha _0} + {\alpha _1}x_1 + {\alpha
_2}x_2,\frac{{\beta _0} + {\beta _1}x_1 }{{\gamma _0} + {\gamma
_1}x_1} \right),\,(\gamma_1,\alpha_2) \neq (0,0).\end{equation}
\begin{itemize}
\item [(a)] Assume that $\alpha_1\ne 0.$ Then the dynamical degree of $F$ is $\delta(F)=\delta^*$ and $d_{n+2}=d_n+d_{n+1}.$
\item[(b)] Assume that $\alpha_1=0.$ Then after an affine change of coordinates $f(x,y)$ takes the form:
\begin{equation}\label{moebius}
f(x_1,x_2)=\left(x_2,\frac{{\beta _0}} {{\gamma _0} + x_1}
\right)\end{equation} and the dynamical degree of $F$ is $\delta(F)=1.$ Furthermore:
\begin{itemize}
\item [($b_1$)] If $h(z):=\frac{{\beta _0}} {{\gamma _0} + z}$ is not a periodic map then $d_n=2$ for all $n\in\N.$
\item [($b_1$)] If $h(z)$ is a $k-$periodic map then $d_n$ is a $2k-$periodic sequence.
\end{itemize}
\end{itemize}
\end{theorem}

\begin{proof}

        To prove $(a)$ we observe that now we have $S_1 \twoheadrightarrow A_1 = O_0=[0:0:1]$
        and
        $F(A_0)=[0:\alpha_1\gamma_1:0]=A_0\notin \mathcal{I}(F).$ So we have
        to blow-up $A_1=[0:0:1]$ getting $E_1.$  Then $\tilde{F}$ sends
        $S_1\rightarrow E_1\twoheadrightarrow [0:1:0]=A_0.$ Since $A_1\in
        S_1\,,\,\pi^*(S_1)=\hat{S_1}+E_1$  and the matrix of
        $\widetilde{F}^ *$ is:
        \begin{equation}\left(
        \begin{array}{ccc}
        2 & 1 \\
        -1 & -1
        \end{array}
        \right).
        \end{equation}

 Then the characteristic polynomial associated to $F$ is $z^2-z-1.$ Hence the dynamical degree
        is $\delta(F) = \delta^*$ and $d_{n+2} = d_{n+1}+d_n$ for all $n \in \N.$

We are going to prove $(b).$ When $\alpha_1=0$ (\ref{t12}) can be
transformed in (\ref{moebius}) via the conjugation
$$\psi(x,y)=\left(\frac{1}{\gamma_1}x+\frac{\alpha_0\gamma_1+\alpha_2\beta_1}{\gamma_1},\frac{1}{\gamma_1
\alpha_2}y+\frac{\beta_1}{\gamma_1}\right).$$ From (\ref{moebius})
we have that $f(f(x,y))=(h(x),h(y))$ and generally:
        \begin{equation}{\label{parells}}f^{2n}(x,y)=\left(h^n(x),h^n(y)\right)\,\,,\,\,f^{2n+1}(x,y)=\left(h^n(y),h^{n+1}(x)\right).\end{equation}

        From the same arguments as before if $h$ is not periodic $d_n=2$ for all $n\in\N.$ If $h$ is $k-$periodic then $f^{2k}(x,y)=(x,y)$ and from (\ref{parells}) we get that $d_n=2$ for all $n\in \N$ such that it is not a multiple of $2k$ and $d_n=1$ for all  $n\in \N$ such that it is a multiple of $2k.$
         In any case case the dynamical degree of $F$ is $\delta(F)=1.$

\end{proof}

\begin{proposition}\label{zeroentropyfive}
Consider the family of mappings:
$$f(x,y)=\left(y,\frac{\beta_0}{\gamma_0+x}\right).$$ Then
\begin{itemize}
\item[(a)] If $\gamma_0^2+4\,\beta_0 \neq 0,$  let $p$ and $q$ be the two different roots of $z^2-\gamma_0\,z-\beta_0 = 0,$ and let $m$ such that $m^2=q/p.$
Then $f(x,y)$  preserves the generically transverse fibrations
        $$H_1(x,y) = \frac{m^2\,p^2+mpx+p(m^2-m+1)y+xy}{(x+p)\,(y+p)},$$
        $$H_2(x,y) = \frac{m^2\,p^2-mpx+p(m^2+m+1)y+xy}{(x+p)\,(y+p)}$$
        with $H_1(f(x,y))=mH_1(x,y)\,,\,H_2(f(x,y))=-mH_2(x,y).$ Furthermore $f(x,y)$ is $2k-$ periodic if and only if $m^{2k}=1$ and in this case $H_1^{2k}(x,y)$ and $H_2^{2k}(x,y)$ are two independent first integrals of $f(x,y).$

  \item[(b)] If $\gamma_0^2+4\,\beta_0 = 0$  then it preserves the two generically transverse fibrations
        $$K_1(x,y)=\frac{\gamma_0^2-2\gamma_0\,x+6\gamma_0\,y+4\,x\,y}{(2\,x+\gamma_0)\,(2\,y+\gamma_0)}\,,\,
        K_2(x,y) = \frac{2\,\,\gamma_0\,(x+y+\gamma_0)}{(2\,x+\gamma_0)\,(2\,y+\gamma_0)},$$
with $K_1(f(x,y))=-K_1(x,y)\,,\,K_2(f(x,y))=K_2(x,y)+1.$ Furthermore $f(x,y)$ is integrable being $W(x,y)=(K_1(x,y))^2$ a first integral.

\end{itemize}
\end{proposition}

\begin{proof}
When $\gamma_0^2+4\,\beta_0 \neq 0$ some calculations give that in
fact $H_1(f(x,y))=mH_1(x,y)$ and $H_2(f(x,y))=-mH_2(x,y).$
Furthermore $H_1(x,y),H_2(x,y)$ are generically transverse because
the determinant of the Jacobian of $H_1(x,y),H_2(x,y)$ is
$$-\frac{2p^2m(m^2-1)}{(p+x)^2(p+y)^2}$$
which is different from zero (if not $m^2=1$ and it happens if and
only if $p=q,$ which is in contradiction with $\gamma_0^2+4\,\beta_0
\neq 0$).

Also when $\gamma_0^2+4\,\beta_0=0$ the determinant of the Jacobian
of $K_1(x,y),K_2(x,y)$ is different from zero because it is equal
to:
$$\frac{16c^2}{(2y+c)^2(2x+c)^2}.$$
Finally $W(x,y)=(K_1(x,y))^2$ is a first integral integral of
$f(x,y)$ because $W(f(x,y))=(K_1(f(x,y)))^2=(-K_1(x,y))^2=W(x,y).$
\end{proof}

\begin{remark} Simple computations give that when $\gamma_0^2+4\,\beta_0 \neq 0,$  $f$ is birationally conjugated to $(mx,-my)$ via the conjugation
$\varphi(x,y)=(H_1(x,y),H_2(x,y))$ and that when
$\gamma_0^2+4\,\beta_0 = 0,$ $f$ is birationally conjugated to
$(-x,y+1)$ via $\psi(x,y)=(K_1(x,y),K_2(x,y)).$
\end{remark}


\begin{thebibliography}{99}
\bibitem{ADMV} Angles, JC, Maillard, JM and Viallet, C. \textsl{On the complexity of some birational transformations,}
{J. Phys. A: Math. Gen.} {\bf 39} (2006) 3641�3654.

\bibitem{BHPV}Barth, W., Hulek K., Peters C. and Ven, A. \textsl{Compact Complex Surfaces.}
Springer Vol. 4 (2004)


\bibitem{BK2}Bedford, E. and Kim, K. \textsl{Periodicities in Linear Fractional Recurrences: Degree Growth of Birational Surface Maps,} Michigan Math. J.
 {\bf 54} (2006),  647--670.

\bibitem{BK3}Bedford, E. and Kim, K. \textsl{The dynamical degrees of a mapping,} Proceedings of the Workshop Future Directions in Difference
Equations, Colecc. Congr. {\bf 69} Univ. Vigo (2011),  3--13.

\bibitem{BV} Bellon, M.P. and  Viallet, C.M. {\sl Algebraic entropy}, {Comm. Math. Phys.}
 {\bf 204(2)}, (1999), 425--437.

\bibitem{BD} Blanc, J. and Deserti, J. \textsl{Degree growth of birational maps of the plane}, (2012), arXiv:1109.6810 [math.AG]

\bibitem{BC} Blanc, J. and Cantat, S.\textsl{Dynamical degree
of birational transformations of projective surfaces,} (2013),
arXiv:1307.0361 [math.AG]

\bibitem{CZ1} Cima A. and Zafar, S. \textsl{Invariant fibrations for some birational maps of $\C^2$}, arXiv:1702.00959
\bibitem{CZ2} Cima A. and Zafar, S. \textsl{Dynamical classification of a Family of Birational Maps of $\C^2$.}, In preparation.
\bibitem{CZ} Cima A. and Zafar, S. \textsl{Integrability and algebraic entropy of k-periodic non-autonomous
    Lyness recurrences.}, {J. Math. Anal. Appl.}, {\bf 413}, (2014),
20-34.



\bibitem{CGM1} Cima, A, Gasull A. and Ma$\tilde{n}$osa, V. \textsl{Dynamics of some rational discrete dynamical systems via invariants}, {J. of Bif. and Ch.},
 {\bf 16(3)}, (2006), 631-645.

\bibitem{CGM2} Cima, A, Gasull A. and Ma$\tilde{n}$osa, V. \textsl{On 2- and 3-periodic Lyness difference equations}, {J. Difference Equ. Appl.},
 {\bf 18(5)}, (2012), 849-864.

\bibitem{CGM3} Cima, A, Gasull A. and Ma$\tilde{n}$osa, V. \textsl{Non-autonomous 2-periodic Gumovski-Mira difference equations}, {Internat. J. Bifur. Chaos Appl.
Sci. Engrg.},  {\bf 22(11)}, (2012), 1250264 (14 pages).


\bibitem{DS}Devault, R. Schultz, S.W. \textsl{On the dynamics of}
$x_{n+1}=\frac{\beta\,x_n +\gamma\,x_{n-1}}{B\,x_n
+D\,x_{n-2}}$,Commun. Appl. Nonlinear Anal. {\bf 12}, (2005), 35�40
.

\bibitem{JD} J. Diller. {\sl Dynamics of Birational Maps of $\pr$}, {Indiana Univ. Math. J.}
{\bf 45, 3}, (1996), 721--772.

\bibitem{DF}Diller, J. and Favre, C. \textsl{Dynamics of bimeromorphic maps of surfaces,} {Amer. J. Math.} {\bf 123} (2001),  1135--1169.


\bibitem{FS}Fornaes, J-E and Sibony, N.  \textsl{Complex dynamics in higher dimension. II,} Modern methods in complex analysis (Princeton, NJ, 1992),
135�182, Ann. of Math. Stud., 137, Princeton Univ. Press, Princeton,
NJ, 1995. Michigan Math. J. {\bf 54} (2006),  647--670.

\bibitem{LG}Ladas, G. \textsl{On the rational recursive sequence} $x_{n+1}=\frac{a+\beta\, x_{n }+ x_{n-1}}{A+B\,x_n +C\,x_{n-1}}$ , {J. Differ. Equ.
Appl.},
 {\bf 1}, (1995) 317�321.

\bibitem{LKP}Ladas, G. Kulenovic, M. and Prokup, N. \textsl{On the Recursive Sequence} $x_{n+1}=\frac{a\,x_n+b\,x_{n-1}}{A+x_n}$,
Journal of Difference Equations and Applications, {\bf 6} , 563-576,
(2000).

\bibitem{LGK}Ladas, G. Gibbons, C.H. and Kulenovic, M.R.S. \textsl{On the rational recursive sequence}
$x_{n+1}=\frac{\alpha+\beta\, x_{n }+ \gamma\,x_{n-1}}{A+B\,x_n}$.
In: Proceedings of the Fifth International Conference on Difference
Equations and Applications, Temuco, Chile, 37 January 2000, pp.
141158. Taylor and Francis, London (2002)

\bibitem{LKMR}Ladas, G. Kulenovic, M.R.S, Martins, L.F. and Rodrigues, I.W. \textsl{On the Dynamics of} $x_{n+1}=\frac{a +b\,x_n}{A+B\,x_n+C\,x_{n-1}}$,
 Facts and Conjectures, Computers and Mathematics with Applications , {\bf 45},  (2003),  1087-1099.

\bibitem{PRo}Pettigrew, J. and Roberts, J.A.G. \textsl{Characterizing singular curves
in parametrized families of biquadratics,} {J. Phys. A} {\bf 41
(11)} (2008),  115203, 28 pp.


\bibitem{Ro1}Roberts, J.A.G. \textsl{Order and symmetry in birational difference
equations and their signatures over finite phase spaces,}
Proceedings of the Workshop Future Directions in Difference
Equations, Colecc. Congr. {\bf 69} Univ. Vigo (2011),  213-221.

\bibitem{Yom}Yomdin, Y. \textsl{Volume growth and entropy,} Israel J. Math. {\bf 57} (1987),  285--300.

\bibitem{SZ}Zafar, S. \textsl{Dynamical Classification of some birational maps of $\C^2$.} Doctoral thesis, UAB, 2014.

\bibitem{ZE}Zayed, E.M.E. and El-Moneam, M.A. \textsl{On the rational recursive sequence} $x_{n+1}=\frac{a\,x_n +b\,x_{n-k} }{cx n -d\,x_{n-k}}$.
Commun. Appl. Nonlinear Anal. {\bf 15}, (2008), 67�76.


\end{thebibliography}
\end{document}